\documentclass[11pt]{article}
\usepackage{graphicx} 
\usepackage{authblk}
\usepackage{amsthm}
\usepackage{amsmath}
\usepackage{amssymb}
\usepackage{enumitem}
\usepackage{framed}
\usepackage{mathrsfs}
\usepackage{xcolor}

\usepackage[colorlinks,linkcolor=blue,citecolor=blue,urlcolor=blue,hypertexnames=true,backref=page]{hyperref}
\setlist[enumerate]{label=(\roman*),labelindent=1em}
\usepackage[margin=1in]{geometry}

\theoremstyle{definition}
\newtheorem{thm}{Theorem}[section]
\newtheorem{lem}[thm]{Lemma}
 \newtheorem{cor}[thm]{Corollary}
 \newtheorem{dfn}[thm]{Definition}
 \newtheorem{rem}[thm]{Remark}
 \newtheorem{exa}[thm]{Example}
 
 \newcommand{\limp}{\longrightarrow}
 
 \newcommand{\mb}{\mathbf{mbCiw}}
 \newcommand{\plog}{\mathbf{PosLog}}
 \newcommand{\ov}{\overline}
 \newcommand{\f}{\mathcal{F}}
 \newcommand{\fm}{\mathsf{Form}}
 \newcommand{\tr}{\mathrm{Trv}}
 \newcommand{\pow}{\mathcal{P}}
 \newcommand{\nc}{\,\mid\!\sim\,}
 \newcommand{\lang}{\mathcal{L}}
\newcommand{\logl}{\mathscr{L}}
\newcommand{\val}{\mathsf{Val}}
\newcommand{\wt}{\mathrm{Wt}}
\newcommand{\cm}{\mathrm{Cons_M}}
\newcommand{\co}{\mathrm{Cons_O}}
\newcommand{\ctr}{\mathrm{Cntr}}

\title{Possibilistic Logic over a Logic of Formal Inconsistency}
\author[1]{Esha Jain}
\author[1]{Sankha S. Basu\footnote{Email: \href{mailto:sankha@iiitd.ac.in}{sankha@iiitd.ac.in}}}
 \affil[1]{Department of Mathematics\\
  Indraprastha Institute of Information Technology-Delhi\\
  New Delhi, India.}
  \date{September 10, 2026}

\begin{document}

\maketitle

\begin{abstract}
In this article, we have introduced a new possibilistic logic on a logic of formal inconsistency with the aim of developing a possibility theoretic framework to deal with uncertainty and inconsistency meaningfully without leading to a system collapse. We have discussed the syntax and semantics for this logic and have proved the soundness and completeness theorems. A set of new measures of consistency, contradictoriness, and triviality of a set of formulas have been defined. These have then been put to use in an example to show that this framework can provide better means of machine reasoning.
\end{abstract}


\tableofcontents

\section{Introduction}
 Logic serves as a foundation for artificial intelligence (AI). A fundamental challenge in AI is that information and knowledge often come with inconsistency and uncertainty. Thus, it is important for a logic for AI to be able to cope with, i.e., effectively derive conclusions in the presence of, inconsistency and uncertainty. The use of \emph{possibility theory} to deal with uncertain information is one such approach. \emph{Possibilistic logic} is the logical counterpart of possibility theory, and was developed in \cite{DuboisLangPrade1994,DuboisPrade2014}. It is a `weighted logic' that handles uncertainty by associating degrees or weights to logic formulas. A basic possibilistic logic formula is a pair $(\varphi;\,x)$ consisting of a classical logic formula $\varphi$ associated with a certainty level $x\in(0,1]$. So, basic possibilistic logic serves the purpose of dealing with uncertainty but the case of inconsistency remains unresolved. One of the limitations of classical logic that renders it unusable in the face of a contradiction is that any contradiction leads to triviality. In other words, everything follows, i.e., \emph{absolute inconsistency} results, from a \emph{negation inconsistency}. Thus, all contradictions are equivalent in classical logic. A \emph{paraconsistent} logic is designed to overcome this problem. A logic $\logl=\langle\lang,\models_{\logl}\rangle$, where $\lang$ is the set of formulas and $\models_\logl\,\subseteq\pow(\lang)\times\lang$ ($\pow(\lang)$ denotes the power set of $\lang$) is the consequence relation of the logic, is said to be \emph{paraconsistent} if it invalidates a \emph{principle of logical explosion}. A principle of explosion is one that is responsible for the resulting of triviality from a contradiction. This is commonly expressed as follows: for all $\alpha,\beta\in\lang$, $\{\alpha,\neg\alpha\}\models_{\logl}\beta$, and is called \emph{ex contradictione sequitor quodlibet (ECQ)}. A logic $\logl$ is then said to be paraconsistent (with respect to $\neg$), if ECQ fails in it, i.e., there exist $\alpha,\beta\in\lang$ such that $\{\alpha,\neg\alpha\}\not\models_{\logl}\beta$. The \emph{logics of formal inconsistency (LFIs)} form a class of paraconsistent logics that were introduced to handle inconsistency while retaining as much of classical reasoning as possible. LFIs (see \cite{Carnielli2007, CarnielliConiglio2016}) achieve this through a unary \emph{consistency} operator $\circ$ in the language. For any formula $\alpha$, the formula $\circ\alpha$ is intended to mean that ``$\alpha$ is consistent.'' Then, in the presence of $\circ\alpha$ as a hypothesis, negation applied on $\alpha$ behaves classically. In other words, $\{\alpha,\neg\alpha\}$ doesn't explode but $\{\alpha,\neg\alpha,\circ\alpha\}$ does. This is called the \emph{gentle principle of explosion}. We have presented this formally in Definition \ref{def:lfi}. However, an LFI, by itself, is not equipped to handle uncertainty. 
  
 This article aims to combine an LFI with possibility theory. We will develop a possibilistic logic over the LFI $\mb$ (discussed in detail in \cite{CarnielliConiglio2016}). This approach derives its enhanced expressive power from the fact that the consistency of a formula is expressible at the object language level in an LFI.  Similar approaches for handling uncertainty using possibility and necessity measures have been taken in \cite{besnard1994possibility} and \cite{CarnielliBueono-Soler2021} using the LFIs $\mathbf{C_1}$ and \textbf{Cie}, respectively. However, their logical counterparts, i.e., the corresponding possibilistic logics were not developed in these articles.
 
 In this connection, we mention another approach for deriving conclusions under uncertainty. This was introduced in \cite{chakraborty1988use,Chakraborty01011995} and here instead of attaching weights to propositions, the consequence relation is graded. As mentioned in \cite{ChakrabortyDutta2019}, the graded consequence relation $\nc$ is a \emph{fuzzy binary relation} (and in that sense a generalization of the classical consequence relation) from $\pow(\lang)$ to $\lang$, where $\lang$ is the set of formulas constructed over some logical alphabet, as before. Every consequence $\Gamma\nc\alpha$, where $\Gamma\cup\{\alpha\}\subseteq\lang$, is assigned a \emph{degree} or \emph{grade} $a$, which is an element of a complete residuated lattice, e.g., the real interval $[0,1]$ with the usual operations. This is denoted by $gr(\Gamma\nc\alpha)=a$. The approach in this this article is different as we assign weights to formulas instead, which represent certainty or necessity.    

 The article is structured as follows. In Section \ref{sec:lfi}, we give a brief general introduction to LFIs followed by a presentation of the LFI to be used in this paper $\mb$. Next, possibility theory is discussed briefly in Section \ref{sec:poss theory}. The main material of the article is in Section \ref{sec:poslog}, where we introduce a possibilistic logic $\plog$ over $\mb$ and discuss its features. Sections \ref{sec:poslog soundness} and \ref{sec:poslog comp} are devoted to the soundness and completeness, respectively, of $\plog$. A measure of the triviality of a set of possibilistic formulas is introduced in Section \ref{sec:poslog comp}. We continue to introduce some more measures, viz., of the contradictoriness, the meta-level and the object-level consistencies of a set of possibilistic formulas in Section \ref{sec:poslog notions}. These are followed by examples which illustrate the applications of these newly introduced concepts.
 
\section{Logics of Formal Inconsistency}\label{sec:lfi}

Let $\lang$ be the set of formulas generated inductively over a denumerable set of variables $V$ using a finite set of connectives or operators, called the \emph{signature}. A logic $\logl$, with signature $\Sigma$, is then a pair $\logl=\langle\mathcal{L},\vdash_{\mathscr{L}}\rangle$, where $\vdash_{\mathscr{L}}\,\subseteq\pow(\mathcal{L})\times\mathcal{L}$ is the consequence relation of the logic. The subscript $\logl$ is dropped from $\vdash_\logl$ when there is no ambiguity regarding the logic under discussion.

\begin{dfn}\label{dfn:Tarski}
A logic $\logl=\langle\lang,\vdash\rangle$ is called \emph{Tarskian} if it satisfies the following properties. For every $\Gamma\cup\Delta\cup\{\alpha\}\subseteq\lang$,
    \begin{enumerate}[label=(\roman*)]
        \item if $\alpha\in\Gamma$, then $\Gamma\vdash\alpha$ (\emph{Reflexivity});
        \item if $\Gamma\vdash\alpha$ and $\Gamma\subseteq\Delta$, then $\Delta\vdash\alpha$ (\emph{Monotonicity});
        \item if $\Delta\vdash\alpha$ and $\Gamma\vdash\beta$ for all $\beta\in\Delta$, then $\Gamma\vdash\alpha$ (\emph{Transitivity/ Cut}).
    \end{enumerate}
    $\logl$ is said to be \emph{finitary} if for every $\Gamma\cup\{\alpha\}\subseteq\lang$, $\Gamma\vdash\alpha$ implies that there exists a finite $\Gamma_0\subseteq\Gamma$ such that $\Gamma_0\vdash\alpha$.

    $\logl$ is said to be \emph{structural} if for every $\Gamma\cup\{\alpha\}\subseteq\lang$, $\Gamma\vdash\alpha$ implies that $\sigma[\Gamma]\vdash\sigma(\alpha)$, for every substitution $\sigma$ of formulas for variables.

    The logic $\logl$ is called \emph{standard} if it is Tarskian, finitary and structural.
\end{dfn}

 As mentioned earlier, LFIs are paraconsistent logics, i.e., logics with non-explosive negations. Moreover, these logics use a unary \emph{consistency} operator, usually denoted by $\circ$. The following is a simplified definition of an LFI. A more general definition can be found in \cite{Carnielli2007,CarnielliConiglio2016}.

\begin{dfn}\label{def:lfi}
Let $\logl=\langle\mathcal{L},\vdash\rangle$ be a standard logic with a signature containing a negation $\neg$ and a primitive or defined unary consistency operator $\circ$. Then, $\logl$ is said to be a \emph{Logic of Formal Inconsistency (LFI)} with respect to $\neg$ and $\circ$ if the following conditions hold.

\begin{enumerate}[label=(\roman*)]
    \item $\{\varphi,\neg\varphi\}\not\vdash\psi$ for some $\varphi,\psi\in\mathcal{L}$.
    \item There exist $\varphi,\psi\in\mathcal{L}$ such that \begin{enumerate}
        \item $\{\circ\varphi,\varphi\}\not\vdash\psi$;
        \item $\{\circ\varphi,\neg\varphi\}\not\vdash\psi$.
    \end{enumerate}
    \item $\{\circ\varphi,\varphi,\neg\varphi\}\vdash\psi$ for all $\varphi,\psi\in\mathcal{L}$.
\end{enumerate}
\end{dfn}

$\mb=\langle\lang,\vdash_\mb\rangle$ is an LFI with signature $\Sigma=\{\land, \lor, \limp, \neg, \circ\}$, where $\neg$ and $\circ$ are unary operators, and the rest are binary operators. The operators $\neg,\land,\lor,\limp$ are, as usual, called \emph{negation, conjunction, disjunction}, and \emph{implication}. The $\circ$ operator is the \emph{consistency} operator -- the intended interpretation of $\circ\alpha$, where $\alpha\in\lang$ is any formula, is that ``$\alpha$ is consistent.'' For more on the consistency operator, see \cite{Carnielli2007,CarnielliConiglio2016}. The logic $\mb$ is induced by the following Hilbert system.

\textsc{Axiom Schema}

\begin{enumerate}
    \item $\alpha\limp(\beta\limp\alpha)$
    \item $(\alpha\limp(\beta\limp\gamma))\limp((\alpha\limp\beta)\limp(\alpha\limp\gamma))$
    \item $\alpha\limp(\beta\limp(\alpha\land\beta))$
    \item $(\alpha\land\beta)\limp\alpha$
    \item $(\alpha\land\beta)\limp\beta$
    \item $\alpha\limp(\alpha\lor\beta)$
    \item $\beta\limp(\alpha\lor\beta)$
    \item $(\alpha\limp\gamma)\limp((\beta\limp\gamma)\limp(\alpha\lor\beta)\limp\gamma)$
    \item $(\alpha\limp\beta)\lor\alpha$
    \item $\alpha\lor\neg\alpha$
    \item \label{ax:bc1}$\circ\alpha\limp(\alpha\limp(\neg\alpha\limp\beta))$
    \item \label{ax:ciw}$\circ\alpha\lor (\alpha\land\neg\alpha)$
\end{enumerate}

\textsc{Inference Rule}

     $\begin{array}{lcl}
          \alpha&&\alpha\limp\beta\\
          \hline
          &\beta
     \end{array}\qquad$[Modus ponens (MP)]

\begin{rem}
    \begin{enumerate}
        \item The logic $\mb$ is the minimal extension of the basic LFI $\mathbf{mbC}$ that guarantees that the truth values of $\alpha$ and $\neg \alpha$ completely determine the truth value of $\circ\alpha$. This is the not the case in $\mathbf{mbC}$. Axioms of $\mathbf{mbC}$ can be obtained by deleting the axiom \ref{ax:ciw} from the above axiom schema. Contradiction is a sufficient, but not necessary, condition for inconsistency in $\mb$. Thus, it distinguishes between inconsistency and contradictoriness, which is one of the important features of LFIs. As mentioned earlier, more details about the LFI $\mb$ can be found in \cite{CarnielliConiglio2016}. 

        \item As a result of axiom \ref{ax:bc1}, a bottom formula $\bot$ is always definable in $\mb$ from any given formula $\beta$ as follows: $\bot:=\beta\land(\neg\beta\land\circ\beta)$. For any formula $\alpha$, we can then define $\sim \alpha$, as $\sim\alpha:=\alpha\limp\bot$. The unary operator $\sim$ behaves like classical negation. This is discussed in detail in \cite{CarnielliConiglio2016}.
    \end{enumerate}
\end{rem}

Given $\Gamma\cup\{\varphi\}\subseteq\lang$, $\varphi$ is said to be a \emph{syntactic consequence} of $\Gamma$, written as $\Gamma\vdash_\mb\varphi$, if there is a \emph{derivation} of $\varphi$ from $\Gamma$, i.e., a finite sequence $(\varphi_1,\ldots,\varphi_n)$ of elements in $\lang$ such that $\varphi_n=\varphi$, and for each $1\le i\le n$, $\varphi$ is either an instance of an axiom, or a member of $\Gamma$, or obtained by an application of MP on $\varphi_j,\varphi_k$, where $1\le j,k< i$. If $\Gamma=\emptyset$, then $\varphi$ is called a \emph{theorem}.

\begin{rem}
    The Deduction theorem holds in $\mb$, i.e., for $\Gamma\cup\{\alpha,\beta\}\subseteq\lang$, $\Gamma\cup\{\alpha\}\vdash_\mb\beta$ iff $\Gamma\vdash_\mb\alpha\limp\beta$. This can be proved using the axioms (i) and (ii) and modus ponens.
\end{rem}
    
We next come to the semantic characterization of $\mb$. To that end, we first define $\mb$-valuations as follows.

\begin{dfn}
    A function $v:\lang\to\{0,1\}$ is an \emph{$\mb$-valuation} if it satisfies the following clauses.
          \begin{enumerate}
              \item $v(\alpha\land\beta)=1$ iff $v(\alpha)=1$ and $v(\beta)=1$;
              \item $v(\alpha\lor\beta)=1$ iff $v(\alpha)=1$ or $v(\beta)=1$;
              \item $v(\alpha\limp\beta)=1$ iff $v(\alpha)=0$ or $v(\beta)=1$;
              \item $v(\neg\alpha)=0$ implies $v(\alpha)=1$;
              \item $v(\circ\alpha)=1$ iff $v(\alpha)=0$ or $v(\neg\alpha)=0$.
          \end{enumerate}
\end{dfn}

An $\mb$-valuation $v$ is said to \emph{satisfy} $\varphi\in\lang$, if $v(\varphi)=1$. $v$ is said to satisfy $\Gamma\subseteq\lang$ if it satisfies every member of $\Gamma$. An $\mb$-formula $\varphi$ (respectively, a set $\Gamma\subseteq\lang$) is said to be \emph{satisfiable}, if there exists an $\mb$-valuation $v$ that satisfies it.

Given $\Gamma\cup\{\varphi\}\subseteq\lang$, $\varphi$ is said to be a \emph{semantical consequence} of $\Gamma$, written as $\Gamma\models_\mb\varphi$, if for every $\mb$-valuation $v$, $v$ satisfies $\varphi$, whenever $v$ satisfies $\Gamma$.

The logic $\mb$ is sound and complete with respect to the above mentioned semantics, i.e., for $\Gamma\cup\{\alpha\}\subseteq\lang$, $\Gamma\vdash_\mb\alpha$ iff $\Gamma\models_\mb\alpha$. The proofs can be found in \cite[Chapters 2, 3]{CarnielliConiglio2016}.

\begin{rem}\label{rem:unsat=trv}
    It can be proved that for any $\Gamma\subseteq\lang$, $\Gamma$ is not satisfiable iff $\Gamma\models_\mb\psi$, and hence, by completeness, $\Gamma\vdash_\mb\psi$, for all $\psi\in\lang$. Thus, a set of $\mb$-formulas is not satisfiable iff it is trivial.
\end{rem}

\begin{rem}
    As mentioned in the introduction, the LFIs $\mathbf{C_1}$ and $\mathbf{Cie}$ have been used for similar studies on handling uncertainty via possibility and necessity measures in \cite{besnard1994possibility} and \cite{CarnielliBueono-Soler2021}, respectively. This begs for a comparison between these logics and $\mb$. In $\mathbf{C_1}$, $\circ\alpha$ is an abbreviation for the formula $\neg(\alpha\land\neg \alpha)$, thus fading the difference between the notions of non-contradictoriness and consistency. The logic $\mathbf{Cie}$ also does not distinguish between inconsistency and contradiction. The logic $\mb$, on the other hand, is equipped with the ability to make this distinction. This is our primary reason for considering $\mb$ here. One might think that the basic LFI $\mathbf{mbC}$ could be used instead. But $\mathbf{mbC}$ is too weak semantically in the sense that truth values of $\alpha$ and $\neg\alpha$ are not sufficient to determine the value of $\circ\alpha$.  
\end{rem}

\section{Possibility Theory}\label{sec:poss theory}

The  theory of possibility originated in \cite{Zadeh1978} based on the theory of fuzzy sets. The concept of a \emph{possibility distribution} is defined as a fuzzy restriction on the values that may be assigned to a variable. Then, the author links a \emph{possibility measure} to the possibility distribution. In \cite{DuboisLangPrade1994}, the authors consider a possibility distribution over the set of interpretations that formally represents a fuzzy set. 

In this article, we will also consider a possibility distribution function over the set of interpretations (or valuations) of a given logic, but we will not see it as a fuzzy set. As mentioned in \cite[Chapter 1]{Hajek1998}, fuzziness is imprecision or vagueness and the truth of a fuzzy proposition is a matter of degree. This is different from uncertainty which is measured with the help of probability and possibility. Probability and possibility are, however, two different concepts. An event might be possible but its probability could be zero. For example, there is always a possibility of rain at a location but its probability could be anywhere between zero and one. On the other hand, if an event is impossible, then its probability must be zero. For example, it is impossible that one would see a giant dinosaur now. Hence, the probability of spotting a giant dinosaur is zero.

Moreover, probability is self-dual, i.e., for any event $\varphi$, `not $\varphi$ is not probable' is equivalent to saying that `$\varphi$ is probable.' This is not the case in possibility. Thus, `it is not possible that not $\varphi$' does not imply that `it is possible that $\varphi$', but `it is necessary that $\varphi$' \cite{CarnielliBueono-Soler2021}. Possibility theory uses a pair of dual functions: possibility and necessity measures unlike probability theory. To further draw out the distinction between possibility and probability, the ideas of \emph{probability qualification}, and \emph{possibility qualification} have been discussed in \cite{Zadeh1978}. Linguistic probability-values describe likelihood, e.g., likely, very likely, very unlikely, etc.; while linguistic possibility values describe to what extent something is possible, e.g., possible, quite possible, slightly possible, impossible, etc. 

We will understand a possibility distribution here as a function measuring possibility over the scale $[0,1]$. The interval $[0,1]$ can be further generalized to any totally ordered set. 

\section{Possibilistic logic}\label{sec:poslog}

We will now develop a possibilistic logic over the LFI $\mb=(\lang,\vdash_\mb)=(\lang,\models_\mb)$ that we have described earlier in Section \ref{sec:lfi}. Let $\val$ be the set of all $\mb$-valuations.

The infimum and supremum of the empty set do not exist, in general. In case of real numbers, they are often defined as $\infty$ and $-\infty$, respectively. Since we are working with the interval $[0,1]$ in this article, we define $\inf\emptyset=1$ and $\sup\emptyset=0$.

\subsection{Language of PosLog}

A \emph{necessity-valued formula} is a pair $(\varphi;\,x)$ where $\varphi\in\lang$ is a formula of the logic $\mb$ and $x\in(0,1]$. The language of $\plog$, the possibilistic logic over $\mb$, consists of necessity-valued formulas. The pair $(\varphi;\,x)$ then expresses that $\varphi$ is certain at least up to degree $x$. This $x$ is called the \textit{weight} of the formula $\varphi$. This is called valuation in \cite{DuboisLangPrade1994} but here we reserve `valuation' for something else. It is to be noted that weighted formulas of the form $(\varphi;\,0)$ are not written as they do not contain any information. Propositional variables and weights are written with English letters $p,q,r,\ldots$ and $x,y,z,\ldots$, respectively, while the rest of the formulas of $\mb$ are expressed by Greek letters. The formulas of $\plog$, i.e., the necessity-valued formulas are also called \emph{possibilistic formulas}. The set of all possibilistic formulas is denoted by $\fm$.

Suppose $\f$ is a set of possibilistic formulas. Then, the \emph{projection} $\f^*$ of $\f$ is the set of $\mb$-formulas obtained from $\f$ by ignoring the weights, i.e., $\f^*=\{\varphi\in\lang\mid\,(\varphi;\,x)\in\f\}$. 

Next, we define the \emph{$x$-cut} and the \emph{strict $x$-cut} of $\f$, denoted respectively, by $\f_{x}$ and $\f_{\ov{x}}$, as follows.
\[
\begin{array}{rcl}
\f_{x}&=&\{(\varphi;\,y)\in\f\mid\,y\geq x\};\\
\f_{\ov{x}}&=&\{(\varphi;\,y)\in\f\mid\,y>x\}.
\end{array}
\]
The projections $\f_{x}^{*}$ and $\f_{\ov{x}}^{*}$ of these are then as follows.
\[
\begin{array}{rcl}
     \f_{x}^{*}&=&\{\varphi\in\lang\mid\,(\varphi\;y)\in\f\mbox{ and } y\geq x\};\\
     \f_{\ov{x} }^{*}&=&\{\varphi\in\lang\mid\,(\varphi\;y)\in\f\mbox{ and }y>x\}. 
\end{array}
\]

\subsection{Semantics}

A \textit{possibility distribution} $\pi$ over the set of $\mb$-valuations $\val$ is a map from $\val$ to $[0,1]$. Intuitively speaking, $\pi$ assigns a value to each $\mb$-valuation, that represents the possibility of that valuation or state.

\begin{dfn}\label{def:poss'meas}
The \textit{possibility measure} $\Pi$ induced by a possibility distribution $\pi$ is a function from $\lang$ to $[0,1]$, defined as follows.
\[
\Pi(\varphi)=\left\{\begin{array}{ll}
\sup\{\pi(w)\mid\,w(\psi)=0,w\in\val\},&\hbox{if $\varphi$ is of the form $\neg\psi$},\\
\sup\{\pi(w)\mid\,w(\varphi)=1,w\in\val\},&\hbox{otherwise}.
\end{array}\right.
\]
\end{dfn}
 
The underlying possibility distribution $\pi$ assigns a value to each valuation or state representing its intrinsic possibility. Then, the  possibility measure $\Pi$ of any given formula is evaluated as the degree of plausibility of the best possible valuation or state in which the formula is true. A formula is thus considered as possible as the most possible valuation that verifies it. The truth of not-$\varphi$ is interchangeable with the falsity of $\varphi$ in classical possibility theory. Here, however, that is not the case as in $\mb$, the behavior of negation is not classical. In $\mb$, for any $w\in\val$ and $\varphi\in\lang$, $w(\varphi)=0$ implies that $w(\neg \varphi)=1$ but the converse is not true. Thus, $\varphi$ and $\neg\varphi$ can both be true simultaneously. 

Defining the possibility measure of a formula of the form $\neg \psi$ using valuations where $\psi$ is false instead of where $\neg\psi$ is true ensures that the possibility of a negated formula considers only the valuations in which $\psi$ is false, and hence, $\neg\psi$ is true (follows from the definition of $\mb$-valuations). This prevents valuations where $\psi$ and $\neg\psi$ are simultaneously true (a possible situation in $\mb$), from being considered, thus avoiding contradictions. So, this condition ensures that the possibility of a negated formula is strictly limited to the failure of the underlying proposition, thus preventing contradictions from distorting the possibility calculations. 

\begin{dfn}\label{def:nec'meas}
The \emph{necessity measure} $N:\lang\to[0,1]$, induced by a possibility distribution $\pi$, is defined as follows.
\[
N(\varphi)= \inf\{1-\pi(w)\mid\,w(\varphi)=0,w\in\val\}
\]
\end{dfn}

Thus, the necessity measure of a formula computes $1-\pi(w)$ for all those valuations $w$ where the formula fails to be true and then takes the infimum of these values.  In case the underlying logic is classical, the necessity measure induced by a possibility distribution $\pi$ is the dual of the possibility measure $\Pi$ induced by $\pi$, and can be defined as $N(\varphi)=1-\Pi(\neg\varphi)$.

\begin{lem}\label{lem: nec'ty axioms proof}
Suppose $\varphi,\psi\in\lang$ are formulas of $\mb$. Let $\pi$ be a possibility distribution and $N$ be the necessity measure induced by $\pi$. Then the following are true.

\begin{enumerate}[label=(\roman*)]
    \item \label{lem:com'son part} If $\varphi\models_{\mb}\psi$ then $N(\varphi)\leq N(\psi)$.
    \item \label{lem:eq-nec} If $\varphi\models_\mb\psi$ and $\psi\models_\mb\varphi$, then $N(\varphi)=N(\psi)$.
    \item \label{lem:tautology} If $\models_\mb\varphi$, then $N(\varphi)=1$.
    \item \label{lem: and min part}$N(\varphi\land\psi)=\min\{N(\varphi),N(\psi)\}$.
\end{enumerate}
\end{lem}

 \begin{proof}
 Let $A=\{1-\pi(w)\mid\,w(\varphi)=0,w\in\val\}$ and $B=\{1-\pi(w)\mid\,w(\psi)=0,w\in\val\}$.
\begin{enumerate}
 \item  Suppose $\varphi\models_{\mb} \psi$. Let $w\in\val$ such that $w(\psi)=0$. Then, $w(\varphi)=0$, which implies that $B\subseteq A$. So, $\inf A\le\inf B$. Hence, $N(\varphi)\leq N(\psi)$. 

 \item This follows from part \ref{lem:com'son part}.

 \item Since $\models_\mb\varphi$, there does not exist a $w\in\val$ such that $w(\varphi)=0$. Hence, $A=\emptyset$, and so, by our assumed convention, $N(\varphi)=\inf A=1$.

 \item Let $C=\{1-\pi(w)\mid\,w(\varphi\land\psi)=0,w\in\val\}$. 

 Suppose $w\in\val$ such that $w(\varphi\land\psi)=0$. Then, by definition of an $\mb$-valuation, either $w(\varphi)=0$ or $w(\psi)=0$. This implies that $C= A\cup B$. Now, by the properties of infimum, $\inf C=\inf A\cup B=\min\{\inf A, \inf B\}$. Thus, $ N(\varphi\land\psi)=\min\{N(\varphi),N(\psi)\}$.
\end{enumerate}
 \end{proof}

\begin{dfn}
    A possibility distribution $\pi$ over the set of $\mb$-valuations $\val$ is said to \textit{satisfy} a possibilistic formula $\Phi=(\varphi;\, x)\in\fm$, iff $N(\varphi)\geq x$, where $N$ is the necessity measure induced by $\pi$. This is denoted by $\pi\models\Phi$.
 
 A possibility distribution $\pi$ over $\val$ is said to satisfy a set of possibilistic formulas $\f\subseteq\fm$, written as $\pi\models\f$, iff  for all $\Phi\in\f$, $\pi\models\Phi$.

 Finally, $\Phi\in\fm$ is said to be a logical consequence of $\f\subseteq\fm$, iff any possibility distribution function that satisfies $\f$ also satisfies $\Phi$, i.e., for all $\pi$, if $\pi\models\f$ then $\pi\models\Phi$. This is denoted by $\f\models\Phi$.

 If $\emptyset\models\Phi$, then we write $\models\Phi$. In this case, it is easy to check that $\pi\models\Phi$, for every possibility distribution $\pi$ over $\val$.
\end{dfn}

\begin{lem}\label{lem:poslog tarskian}
    $\plog$ developed over the logic $\mb$ satisfies the following conditions. For any $\f\cup\mathcal{G}\cup\{\Phi\}\subseteq\fm$,
    \begin{enumerate}
        \item if $\Phi\in \f$, then $\f\models\Phi$ (Reflexivity);
        \item if $\f\models\Phi$ and $\f\subseteq \mathcal{G}$, then $\mathcal{G}\models\Phi$ (Monotonicity);
        \item if $\mathcal{G}\models\Phi$ and $\f\models\Psi$ for all $\Psi\in \mathcal{G}$, then $\mathcal{F}\models\Phi$ (Transitivity).
    \end{enumerate}
\end{lem}

\begin{proof}
    \begin{enumerate}
        \item Suppose $\pi$ is a possibility distribution that satisfies $\f$. Then, $\pi$ satisfies every possibilistic formula in $\f$. So, since $\Phi\in\f$, in particular, $\pi\models\Phi$. Thus, $\f\models\Phi$.
        
        \item Suppose $\pi$ is a possibility distribution that satisfies $\mathcal{G}$. Then, by definition, it satisfies every element of $\mathcal{G}$. Since $\f\subseteq \mathcal{G}$, $\pi$ satisfies all formulas of $\f$, i.e., $\pi\models\f$. Now, as $\f\models\Phi$, this implies that $\pi\models\Phi$. Thus, $\mathcal{G}\models\Phi$.  
        
        \item Suppose $\pi$ is a possibility distribution that satisfies $\f$. Now, since $\f\models\Psi$ for every $\Psi\in\mathcal{G}$, this implies that $\pi\models\Psi$ for every $\Psi\in\mathcal{G}$. Thus, $\pi\models\mathcal{G}$. Then, as $\mathcal{G}\models\Phi$, $\pi\models\Phi$. Hence, $\f\models\Phi$. 
    \end{enumerate}
\end{proof} 

\begin{rem}
    The above lemma shows that $\plog$ satisfies the Tarskian conditions discussed in Definition \ref{dfn:Tarski}.
\end{rem}

\begin{lem}\label{lem:properties}
 Suppose $\varphi\in\lang$ is an $\mb$-formula. Then, the following properties hold. 
 \begin{enumerate}
     \item $(\varphi;\,x)\models(\varphi;\,y)$ for all $0<y\le x\le1$;
     \item $\models(\varphi;\,x)$, for all $x\in(0,1]$, iff $\varphi$ is an $\mb$-tautology.
 \end{enumerate}
\end{lem}

\begin{proof}
    \begin{enumerate}[label=(\roman*)]
        \item Let $\pi$ be a possibility distribution such that $\pi\models(\varphi;\,x)$. Then, $N(\varphi)\ge x$, where $N$ is the necessity measure induced by $\pi$. So, for any $y\le x$, $N(\varphi)\ge y$. Thus, $\pi\models(\varphi;\,y)$.

        \item Suppose $\models_\mb\varphi$. Let $\pi$ be a possibility distribution and $N$ the necessity measure induced by $\pi$. Now, since $\models_\mb\varphi$, by Lemma \ref{lem: nec'ty axioms proof}\ref{lem:tautology}, $N(\varphi)=1$. So, for any $x\in(0,1]$, $N(\varphi)\ge x$, which implies that $\pi\models(\varphi;\,x)$. Since $\pi$ was an arbitrary possibility distribution, this implies that every possibility distribution satisfies $(\varphi;\,x)$. Thus, $\models(\varphi;\,x)$ for all $x\in(0,1]$.

        Conversely, suppose $\models(\varphi;\,x)$ for all $x\in(0,1]$ but $\not\models_\mb\varphi$. Let $\pi:\val\to[0,1]$ be a possibility distribution defined by $\pi(w)=0.5$, for all $w\in\val$. Since $\not\models_\mb\varphi$, there exists $w\in\val$ such that $w(\varphi)=0$. Thus, $\{1-\pi(w)\mid\,w(\varphi)=0,w\in\val\}\neq\emptyset$, and hence, $N(\varphi)=\inf\{1-\pi(w)\mid\,w(\varphi)=0,w\in\val\}=0.5$. This implies that $\pi\not\models (\varphi;\,x)$, for any $x>0.5$, and so, $\not\models(\varphi;\,x)$ for $x>0.5$. This, however, contradicts our assumption. Hence, $\models_\mb\varphi$.
    \end{enumerate}
\end{proof}

\begin{lem}\label{lem:equivalence}
    Suppose $\f\subseteq\fm$ be a set of possibilistic formulas and $\varphi,\psi\in\lang$ be $\mb$-formulas. If $\varphi\models_\mb\psi$ and $\psi\models_\mb\varphi$, then for any $x\in(0,1]$, $\f\models(\varphi;\,x)$ iff $\f\models(\psi;\,x)$.
\end{lem}

\begin{proof}
    Suppose $\f\models(\varphi;\,x)$ and $\pi$ is a possibility distribution such that $\pi\models\f$. Then, $\pi\models(\varphi;\,x)$, which implies that $N(\varphi)\ge x$, where $N$ is the necessity measure induced by $\pi$. Now, since $\varphi\models_\mb\psi$ and $\psi\models_\mb\varphi$, by Lemma \ref{lem: nec'ty axioms proof}\ref{lem:eq-nec}, $N(\varphi)=N(\psi)$. So, $N(\psi)\ge x$, and thus, $\pi\models(\psi;\,x)$. Hence, $\f\models(\psi;\,x)$.

    The converse follows by analogous arguments.
\end{proof}

\begin{dfn}
     Suppose $\f\subseteq\fm$ be a set of possibilistic formulas and $\varphi\in\lang$ be an $\mb$-formula. Then, the \emph{weight of $\varphi$ relative to $\f$}, denoted by $\wt(\varphi,\f)$, is defined as follows.
     \[
     \wt(\varphi,\f)=\mathrm{sup}\{x\in(0,1]\mid\,\f\models(\varphi;\,x)\}
     \]
\end{dfn}

\begin{rem}
    Intuitively speaking, $\wt(\varphi,\f)$ is the greatest value of $x$ such that $(\varphi;\,x)$ is a logical consequence of $\f$. This is denoted by $\mathrm{Val}(\varphi,\f)$ in \cite{DuboisLangPrade1994}.
\end{rem}

\begin{exa} Let $\f=\{(\neg\neg p;\,0.7), (p\lor q;\,0.4)\}$, where $p,q\in V$ are variables. The following table lists the distinct $\mb$-valuations for the variables $p,q$.

\begin{center}
 $\begin{array}{|c|c|c|c|c|c|}
           \hline p& q & \neg p& \neg\neg p&p\lor q&\\\hline
             1&1  &1&1&1&v_1 \\\cline{4-4}
             & & &0  &&v_2\\\cline{3-4}
             &  &0 &1&&v_3\\\hline
             1& 0 & 1&1&1&v_4\\\cline{4-4}
             &&&0  &&v_5\\\cline{3-4}
             &&0&1  &&v_6\\\hline
        0&1&1&1&1&v_7\\\cline{4-4}
        &&&0&&v_8\\\hline
        0&0&1&1&0&v_9\\\cline{4-4}
        &&&0&&v_{10}\\\hline
    \end{array}$
\end{center}
    
     Now, a possibility distribution $\pi$ satisfies $\f$, i.e., $\pi\models\f$ iff $N(\neg \neg p)\geq 0.7$ and $N(p\lor q)\geq 0.4$, i.e., $\inf\{1-\pi(w)\mid\,w(\neg\neg p)=0,w\in\val\}\geq 0.7$ and $\inf\{1-\pi(w)\mid\,w( p\lor q)=0,w\in\val\}\geq 0.4$. Thus, $\pi\models\f$ iff $1-\pi(w)\ge0.7$, i.e., $\pi(w)\leq 0.3$, for all $w\in\val$ such that $w(\neg\neg p)=0$, and $1-\pi(w)\ge0.4$, i.e, $\pi(w)\leq 0.6$, for all $w\in\val$ such that $w(p\lor q)=0$.

     Now, we note from the table above that $w(\neg\neg p)=0$ for $w=v_2,v_5,v_8$, and $v_{10}$, while $w(p\lor q)=0$ for $w=v_9$. Thus, $\pi\models\f$ iff $\pi(w)\le0.3$ for $w=v_2,v_5,v_8,v_{10}$ and $\pi(v_9)\le0.6$.
\end{exa}

\begin{rem}
    It may be noted that in classical logic, there are fewer valuations due to more equivalences among formulas. Thus, by using $\mb$ instead of classical logic, we get finer separations via the valuations.
\end{rem}

\begin{dfn}\label{def:posdist'on on formulae}
     Let $\f\subseteq\fm$ be a set of possibilistic formulas. The \emph{possibility distribution induced by $\f$} is then defined as follows. For any $w\in\val$,
\[
\pi_{\f}(w)=\inf\{1-x\mid\,(\varphi;\,x)\in\f,w(\varphi)=0\}.
\]
\end{dfn}

\begin{lem}\label{lem:inducedposs'distr}
    Suppose $\f\subseteq\fm$ and $\pi_\f$ be the possibility distribution induced by $\f$. Then, $\pi_\f\models\f$.
\end{lem}

\begin{proof}
    Suppose $(\varphi;\,x)\in\f$. We need to show that $\pi_\f\models(\varphi;\,x)$, i.e., $N_\f(\varphi)\ge x$, where $N_\f$ is the necessity measure induced by $\pi_\f$.

    \textsc{Case 1:} Suppose $w(\varphi)=1$ for all $w\in\val$.

    Then, $\{1-\pi_\f(w)\mid\,w(\varphi)=0, w\in\val\}=\emptyset$, and hence, $N_\f(\varphi)=\inf\{1-\pi_\f(w)\mid\,w(\varphi)=0,w\in\val\}=\inf\emptyset=1\ge x$.

    \textsc{Case 2:} Suppose there exists $w\in\val$ such that $w(\varphi)=0$.

     Now, $\pi_\f(w)=\inf\{1-y\mid\,(\psi,y)\in\f,w(\psi)=0\}\le1-x$, i.e., $1-\pi_\f(w)\ge x$. Thus, $1-\pi_\f(w)\ge x$ for all $w\in\val$ such that $w(\varphi)=0$. Hence, $N_\f(\varphi)=\inf\{1-\pi_\f(w)\mid\,w(\varphi)=0,w\in\val\}\ge x$.

    Thus, in all cases, $N_\f(\varphi)\ge x$, which implies that $\pi_\f\models(\varphi;\,x)$. Since $(\varphi;\,x)\in\f$ was arbitrary, $\pi_\f\models\f$.
\end{proof}

The following theorem presents a necessary and sufficient condition for a set of possibilistic formulas to be satisfied by a possibility distribution function.

 \begin{thm}\label{thm: pi satisfies f}
  Suppose $\pi$ is a possibility distribution over the set of $\mb$-valuations $\val$ and $\f\subseteq\fm$. Then, $\pi\models\f$ if and only if $\pi\leq\pi_{\f}$, i.e., for all $w\in\val, \pi(w)\leq\pi_{\f}(w)$.
 \end{thm}
 
 \begin{proof}
     Suppose $\pi\models\f$. Then, $\pi\models(\varphi;\,x)$, for all $(\varphi;\,x)\in\f$. This implies that $N(\varphi)\geq x$, for all $(\varphi;\,x)\in\f$, where $N$ is the necessity measure induced by $\pi$. So, for any $(\varphi;\,x)\in\f$, $\inf\{1-\pi(w)\mid\,w(\varphi)=0,w\in\val\}\geq x$, i.e., $\pi(w)\le1-x$ for all $w\in\val$ such that $w(\varphi)=0$. Thus, $\pi(w)\le\inf\{1-x\mid\,(\varphi;\,x)\in\f,w(\varphi)=0\}=\pi_\f(w)$. Hence, $\pi(w)\le\pi_\f(w)$ for all $w\in\val$.

     Conversely, suppose $\pi\le\pi_\f$. Let $(\varphi;\,x)\in\f$. Since $\pi(w)\le\pi_\f(w)$ for all $w\in\val$, $\inf\{1-\pi_\f(w)\mid\,w(\varphi)=0,w\in\val\}\le\inf\{1-\pi(w)\mid\,w(\varphi)=0,w\in\val\}$. Thus, $N_\f(\varphi)\le N(\varphi)$, where $N_\f$ and $N$ are the necessity measures induced by $\pi_\f$ and $\pi$, respectively. Now, by Lemma \ref{lem:inducedposs'distr}, $N_\f(\varphi)\ge x$. Thus, $N(\varphi)\ge x$, which implies that $\pi\models(\varphi;\,x)$. Since $(\varphi;\,x)\in\f$ was arbitrary, this proves that $\pi\models\f$.
 \end{proof}

\begin{cor}\label{cor:pi_FsatF}
    Suppose $\f\cup\{(\varphi;\,x)\}\subseteq\fm$. Then, $\f\models(\varphi;\,x)$ iff $\pi_\f\models(\varphi;\,x)$, where $\pi_\f$ is the possibility distribution induced by $\f$.
\end{cor}

\begin{proof}
Suppose $\f\models(\varphi;\,x)$. Then, for any possibility distribution $\pi$, if $\pi\models\f$, then $\pi\models(\varphi;\,x)$. By Lemma \ref{lem:inducedposs'distr}, $\pi_\f\models\f$. Thus, $\pi_\f\models(\varphi;\,x)$. 

Conversely, suppose $\pi_\f\models(\varphi;\,x)$. Let $\pi$ be a possibility distribution such that $\pi\models\f$. We need to show that $\pi\models(\varphi;\,x)$. By Theorem \ref{thm: pi satisfies f}, $\pi\leq\pi_\f$, which implies that $1-\pi(w)\geq 1-\pi_\f(w)$ for all $w\in\val$. Thus, $\inf\{1-\pi(w)\mid\, w(\varphi)=0,w\in\val\}\geq \inf\{1-\pi_\f(w)\mid\, w(\varphi)=0,w\in\val\}$. So, $N(\varphi)\geq N_\f(\varphi)\geq x$, where $N$ and $N_\f$ are the necessity measures induced by $\pi$ and $\pi_\f$, respectively. Hence, $\pi\models(\varphi;\,x)$. This implies that if $\pi\models\f$ then $\pi\models(\varphi;\,x)$, for any possibility distribution $\pi$. Thus, $\f\models(\varphi;\,x)$.
\end{proof}

\begin{cor}\label{cor:wt=nec'ity}
    Suppose $\f\subseteq\fm$ and $\varphi$ is an $\mb$-formula. Then, $\wt(\varphi,\f)=N_\f(\varphi)$, where $N_\f$ is the necessity measure induced by the possibility distribution $\pi_\f$ that is induced by $\f$. Moreover, if $N_\f(\varphi)>0$, then $\f\models(\varphi;\,N_\f(\varphi))$.
\end{cor}

\begin{proof}
    \textsc{Case 1:} Suppose $N_\f(\varphi)=0$. Then, we claim that $\f\not\models(\varphi;\,x)$ for all $x\in(0,1]$. 

    If possible, let $\f\models(\varphi;\,x)$. Then, by Corollary \ref{cor:pi_FsatF}, $\pi_\f\models(\varphi;\,x)$. This implies that $N_\f(\varphi)\ge x>0$. This is a contradiction. Hence, $\f\not\models(\varphi;\,x)$ for all $x\in(0,1]$. Thus, $\wt(\varphi,\f)=\sup\{x\in(0,1]\mid\,\f\models(\varphi;\,x)\}=0$. So, $\wt(\varphi,\f)=N_\f(\varphi)$.

    \textsc{Case 2:} Suppose $N_\f(\varphi)>0$. 
    
    Let $\wt(\varphi,\f)=a$ and $N_\f(\varphi)=b$. We first show that $\f\models(\varphi;\,b)$.

    Suppose $\pi$ is a possibility distribution such that $\pi\models\f$. Then, by Theorem \ref{thm: pi satisfies f}, $\pi\le\pi_\f$, i.e., $\pi(w)\le\pi_\f(w)$, for all $w\in\val$. This implies that $1-\pi_\f(w)\le1-\pi(w)$ for all $w\in\val$. So, $\inf\{1-\pi_\f(w)\mid\,w(\varphi)=0,w\in\val\}\le\inf\{1-\pi(w)\mid\,w(\varphi)=0,w\in\val\}$. Thus, $b=N_\f(\varphi)\le N(\varphi)$, where $N$ is the necessity measure induced by $\pi$. Hence, $\pi\models(\varphi;\,b)$. Since $\pi$ was an arbitrary possibility distribution satisfying $\f$, this implies that $\f\models(\varphi;\,b)$, i.e., $\f\models(\varphi;\,N_\f(\varphi))$. 
    
    Thus, $b\in\{x\in(0,1]\mid\,\f\models(\varphi;\,x)\}$, and so, $b\le\sup\{x\in(0,1]\mid\,\f\models(\varphi;\,x)\}=a$.

    Now, to establish that $a=b$, we need to show that $a\le b$.

    \textsc{case 1:} $\{x\in(0,1]\mid\,\f\models(\varphi;\,x)\}=\emptyset$.

    In this case, $\f\not\models(\varphi;\,x)$ for all $x\in(0,1]$. This implies that, for each $x\in(0,1]$, there exists a possibility distribution $\pi$ such that $\pi\models\f$ but $\pi\not\models(\varphi;\,x)$, i.e., $N(\varphi)<x$, where $N$ is the necessity measure induced by $\pi$. Now, by the proof of Theorem \ref{thm: pi satisfies f}, $N_\f(\varphi)\le N(\varphi)$. So, for each $x\in(0,1]$, $N_\f(\varphi)<x$. This implies that $N_\f(\varphi)=0$, which contradicts our assumption that $N_\f(\varphi)>0$. Thus, this case is impossible.

    \textsc{case 2:} $\{x\in(0,1]\mid\,\f\models(\varphi;\,x)\}\neq\emptyset$.

    Let $y\in\{x\in(0,1]\mid\,\f\models(\varphi;\,x)\}$. Then, $\f\models(\varphi;\,y)$, and so, by Corollary \ref{cor:pi_FsatF}, $\pi_\f\models(\varphi;\,y)$. This implies that $N_\f(\varphi)\ge y$. This proves that $b\ge y$ for all $y\in\{x\in(0,1]\mid\,\f\models(\varphi;\,x)\}$, which implies that $b\ge\sup\{x\in(0,1]\mid\,\f\models(\varphi;\,x)\}=a$.

    Hence, $a=\wt(\varphi,\f)=N_\f(\varphi)=b$.
\end{proof}

\begin{rem}\label{rem:N_F=infN}
    From the proof of the above corollary, we can see that  $N_\f(\varphi)$ is a lower bound of the set $\{N(\varphi)\mid\,\pi\models \f,N \hbox{ is the necessity measure induced by }\pi\}$. Since $N_\f(\varphi)\in\{N(\varphi)\mid\,\pi\models \f,N \hbox{ is the necessity measure induced by }\pi\}$, it becomes the infimum. Thus, $N_\f(\varphi)=\inf\{N(\varphi)\mid\,\pi\models\f,N \hbox{ is the necessity measure induced by }\pi\}$. 
     In other words, the best weight $x$ such that $(\varphi;\,x)$ is a logical consequence of $\f$, i.e., $\wt(\varphi,\f)$, is the best lower bound of necessity measures, if this is greater than zero. 
\end{rem}

\begin{cor}\label{cor:non-finitary}
    $\plog$ is not finitary.
\end{cor}

\begin{proof}
    We prove this by the following example.

     Let $\f=\{(p;\,x-1/n)\mid\,n\in\{2,3,\dots\}\}$, where $p\in V$ is a variable of $\mb$ and  $x$ is a fixed real number such that $0.5<x<1$. Clearly, there exists $w\in\val$ such that $w(p)=0$. Thus, $\{1-(x-1/n)\mid\,w(p)=0\}\neq\emptyset$. So, $\pi_\f(w)=\inf\{1-(x-1/n)\mid\,w(p)=0\}=1-x$. This implies that $N_\f(p)=\inf\{1-\pi_\f(w):w(p)=0,w\in\val\}=x\neq0$. Then, by Corollary \ref{cor:wt=nec'ity}, $\wt(p,\f)=N_{\f}(p)=x$, and hence, $\pi_\f\models(p;\,x)$. Thus, by Corollary \ref{cor:pi_FsatF}, $\f\models(p;\,x)$.

    Now, we claim that for any finite $\f^\prime\subseteq\f$, $\f^\prime\not\models(p;\,x)$. 
    
    \textsc{Case 1:} $\f^\prime=\emptyset$.

    We note that $\pi_\emptyset(w)=\inf\emptyset=1$ for any $w\in\val$. Hence, $N_\emptyset(p)=\inf\{1-\pi_\emptyset(w)\mid\,w(p)=0,w\in\val\}=0$. Hence, by Corollary \ref{cor:wt=nec'ity}, $\wt(p,\emptyset)=\sup\{y\in(0,1]\mid\,\emptyset\models(p,y)\}=0$, which implies that $\emptyset\not\models(p;\,x)$ as $x>0$.

    \textsc{Case 2:} $\f^\prime$ is a non-empty finite subset of $\f$.
    
    Suppose $\f^\prime=\{(p;\,x-1/n_1),\ldots,(p;\,x-1/n_k)\}$. Without loss of generality, we assume that the $x-1/n_1\le\cdots\le x-1/n_k$. Thus, $\pi_{\f^\prime}(w)=\inf\{1-(x-1/n_i)\mid\,(p;\,x-1/n_i)\in\f^\prime,w(p)=0\}=1-(x-1/n_k)$. So, $N_{\f^\prime}(p)=\inf\{1-\pi_{\f^\prime}(w)\mid\,w(p)=0,w\in\val\}=x-1/n_k$. Then, by Corollary \ref{cor:wt=nec'ity}, $\wt(p,\f^\prime)=x-1/n_k$, i.e., $\sup\{y\in(0,1]\mid\,\f^\prime\models(p;\,y)\}=x-1/n_k< x$. Hence, $\f^\prime\not\models(p;\,x)$.
\end{proof}

\begin{rem}
Given two possibility distributions, $\pi$ and $\pi^\prime$, $\pi$ is said to be \emph{more specific (more informative or restrictive)} than $\pi^\prime$ if for each valuation $w$, $\pi(w)\leq\pi^\prime(w)$. 

The principle of minimal specificity states that any hypothesis not known to be impossible cannot be ruled out \cite{DuboisPrade2014}. This plays an important role in possibility theory. Given a set of possibilistic formulas, there are usually many different possibility distributions that satisfy it. The formulas can be thought of as evidences with degrees of certainty attached. So, according to the principle of minimal specificity, we should select the least specific distribution compatible with the given evidence. In other words, we should assign a degree of maximum possibility to each valuation/state, taking evidence into account. This principle is obeyed here through the construction of $\pi_\f$ for a given set of formulas $\f$.
\end{rem}

 \begin{lem}\label{lem: comp'on  min nd arrow} Suppose $\pi$ is a possibility distribution and $N$ is the necessity measure induced by it. Then, for any $\mb$-formulas $\varphi, \psi\in \lang $, 
    $N(\psi)\geq \min\{N(\varphi),N(\varphi\limp\psi)\}$. 
 \end{lem}
 
 \begin{proof} We note that $N(\psi)=\inf\{1-\pi(w)\mid\,w(\psi)=0,w\in\val\}=\inf(A\cup B)=\min\{\inf A,\inf B\}$, where $A=\{1-\pi(w)\mid\,w(\psi)=0 \mbox{ and } w(\varphi)=0, w\in\val\}$ and 
           $B=\{1-\pi(w)\mid w(\psi)=0\mbox{ and } w(\varphi)=1, w\in\val\}$. Now, $A=\{1-\pi(w)\mid\,w(\varphi\lor\psi)=0,w\in\val\}$ and $B=\{1-\pi(w)\mid\,w(\varphi\limp\psi)=0,w\in\val\}$. Thus, $\inf A=N(\varphi\lor\psi)$ and $\inf B=N(\varphi\limp\psi)$. Now, as $\varphi\models_\mb \varphi\lor \psi$, by Lemma \ref{lem: nec'ty axioms proof}, $N(\varphi\lor\psi)\geq N(\varphi)$. Hence, $ N(\psi)=\min\{N(\psi\lor\varphi), N(\varphi\limp \psi)\}\geq \min\{N(\varphi), N(\varphi\limp \psi)\}$.
 \end{proof}

 Using the above lemma, we now prove a deduction theorem for $\plog$ below.
 
 \begin{thm}[Deduction Theorem]\label{thm:ded'on thm poslog}
     Suppose $\f\subseteq\fm$, $\varphi,\psi\in\lang$ are $\mb$-formulas, and $x\in(0,1]$. Then,
     $\f\cup\{(\varphi;\,1)\}\models(\psi;\,x)$ iff $\f\models(\varphi\limp \psi;\,x)$.
 \end{thm}
 
 \begin{proof}
     Suppose $\f\models(\varphi\limp\psi;\,x)$. We need to show that $\f\cup\{(\varphi;\,1)\}\models(\psi;\,x)$. Let $\pi$ be a possibility distribution such that $\pi\models\f\cup\{(\varphi;\,1)\}$. Then, $\pi\models\f$ and $N(\varphi)=1$, where $N$ is the necessity measure induced by $\pi$. Now, as $\f\models(\varphi\limp\psi;\,x)$, since $\pi\models\f$, $\pi\models(\varphi\limp\psi;\,x)$, i.e., $N(\varphi\limp\psi)\geq x$. Then, since $N(\varphi)=1$, by Lemma \ref{lem: comp'on  min nd arrow}, $N(\psi)\geq\min\{N(\varphi), N(\varphi\limp \psi)\}=N(\varphi\limp\psi)\ge x$. Thus, $\pi\models(\psi;\,x)$, and hence, $\f\cup\{(\varphi;\,1)\}\models(\psi;\,x)$.

    Conversely, suppose $\f\cup\{(\varphi;\,1)\}\models(\psi;\,x)$. Then, by Corollary \ref{cor:pi_FsatF}, $\pi_{\f\cup\{(\varphi;\,1)\}}\models(\psi;\,x)$, where $\pi_{\f\cup\{(\varphi;\,1)\}}$ is the possibility distribution induced by $\f\cup\{(\varphi;\,1)\}$. So, $N_{\f\cup\{(\varphi;\,1)\}}(\psi)\geq x$, where $N_{\f\cup\{(\varphi;\,1)\}}$ is the necessity measure induced by $\pi_{\f\cup\{(\varphi;\,1)\}}$. This implies that $\inf\{1-\pi_{\f\cup\{\varphi;\,1\}}(w)\mid\,w(\psi)=0,w\in\val\}\geq x$. Next, we note that $\pi_{\f\cup\{(\varphi;\,1)\}}(w)=\pi_\f(w)$, for any valuation $w$ such that $w(\varphi)=1$. So, 
    \[
    \begin{array}{lcl}
         \{1-\pi_\f(w)\mid\,w(\varphi\limp\psi)=0,w\in\val\}&=&\{1-\pi_\f(w)\mid\,w(\varphi)=1,w(\psi)=0,w\in\val\}\\
         &=&\{1-\pi_{\f\cup\{\varphi;\,1\}}(w)\mid\,w(\varphi)=1,w(\psi)=0,w\in\val\}\\
         &\subseteq&\{1-\pi_{\f\cup\{\varphi;\,1\}}(w)\mid\,w(\psi)=0,w\in\val\}.
    \end{array}
    \]
    Thus, 
    \[
    \begin{array}{lcl}
         N_\f(\varphi\limp\psi)&=&\inf\{1-\pi_\f(w)\mid\,w(\varphi\limp\psi)=0,w\in\val\}\\
         &\ge&\inf\{1-\pi_{\f\cup\{\varphi;\,1\}}(w)\mid\,w(\psi)=0,w\in\val\}\\
         &\ge&x.
    \end{array}
    \]
    Hence, by Corollary \ref{cor:pi_FsatF}, $\f\models(\varphi\limp\psi;\,x)$.
 \end{proof}

\section{Hilbert System for \textbf{PosLog} and Soundness}\label{sec:poslog soundness}

We now present the following Hilbert-style system for $\plog$. The axioms in the following list are the axioms of $\mb$ weighted by 1. The first inference rule is known as \textit{graded modus ponens (GMP)}. The second inference rule (S) intuitively says that if a proposition holds with some degree of certainty, then it is also holds with any lesser degree of certainty as well.

\vspace{1mm}
\textsc{Axiom Schema}
\begin{enumerate}
    \item $(\alpha\limp(\beta\limp\alpha);\,1)$
    \item $((\alpha\limp(\beta\limp\gamma))\limp((\alpha\limp\beta)\limp(\alpha\limp\gamma));\,1)$
    \item $(\alpha\limp(\beta\limp(\alpha\land\beta));\,1)$
    \item $((\alpha\land\beta)\limp\alpha;\,1)$
    \item $((\alpha\land\beta)\limp\beta;\,1)$
    \item $(\alpha\limp(\alpha\lor\beta);\,1)$
    \item $(\beta\limp(\alpha\lor\beta);\,1)$
    \item $((\alpha\limp\gamma)\limp((\beta\limp\gamma)\limp(\alpha\lor\beta)\limp\gamma);\,1)$
    \item $((\alpha\limp\beta)\lor\alpha;\,1)$
    \item $(\alpha\lor\neg\alpha;\,1)$
    \item $(\circ\alpha\limp(\alpha\limp(\neg\alpha\limp\beta));\,1)$
    \item $(\circ\alpha\lor (\alpha\land\neg\alpha);\,1)$
\end{enumerate}

\textsc{Inference Rules} 
\begin{enumerate}
    \item $\begin{array}{c}
         (\alpha;\,x),\,(\alpha\limp\beta;\,y)\\
         \hline
         (\beta;\,\min\{x,y\})
    \end{array}$\hspace{0.1in} [GMP]
    \item $\begin{array}{c}
        (\alpha;\,x)\\
        \hline
        (\alpha;\,y)
    \end{array}$, for any $y\leq x$\hspace{0.1in} [S]
\end{enumerate}

 \begin{dfn}\label{def:der'n}
Suppose $\f\cup\{\Phi\}\subseteq\fm$. A \emph{derivation} of $\Phi$ from $\f$  is a finite sequence $(\Phi_{1},\ldots,\Phi_{n})$ of elements in $\fm$, where $\Phi_{n}=\Phi$ and for each $1\le i\le n$, 

\begin{enumerate}[label=(\roman*)]
    \item $\Phi_{i}$ is an instance of an axiom, or
    \item $\Phi_{i}\in \f$, or
    \item there exist $1\le j,k<i$ such that $\Phi_{i}$ is obtained from $\Phi_{j},\Phi_{k}$ by a use of the rule GMP, or
    \item there exists  $j<i$ such that  $\Phi_{i}$ is the result of an application of the rule S on $\Phi_j$.
\end{enumerate}
We say that $\Phi$ is \emph{syntactically derivable} or \emph{syntactically entailed} from $\f$, and write $\f\vdash\Phi$, if there is a derivation of $\Phi$ from $\f$. If $\f=\emptyset$, then $\Phi$ is called a \textit{theorem} of $\plog$. In other words, $\Phi$ is a theorem if it can be derived using only axioms and rules.
\end{dfn}

 We now show that the above Hilbert-style system for $\plog$ is sound with respect to the semantics presented in Section \ref{sec:poslog}.
 
 \begin{thm}[Soundness]\label{thm:soundness}
     Suppose $\f\cup\{(\gamma;\,y)\}\subseteq\fm$ is a set of possibilistic formulas. If 
     $\f\vdash(\gamma;\,y)$ then $\f\models(\gamma;\,y)$.
 \end{thm}
 
 \begin{proof}
     Let $\pi$ be a possibility distribution on the set of $\mb$-valuations $\val$ and $N$ be the necessity measure induced by $\pi$. We first establish the following claims.
     \begin{enumerate}
         \item\label{sd'ss;1} $\pi\models(\varphi;\,1)$ where $\varphi$ is an instance of an $\mb$-axiom.

        We need to show $N(\varphi)\geq 1$. Since $\varphi$ is an instance of an $\mb$-axiom, there is no $w$ such that $w(\varphi)=0$. Thus, $N(\varphi)=\inf\{1-\pi(w)\mid\,w(\varphi)=0, w\in\val\}=\inf\emptyset=1$.
        
        \item\label{sd'ss;2} If $\pi\models(\varphi;\,x)$ and $\pi\models(\varphi\limp\psi;\,y)$, then $\pi\models(\psi;\,\min\{x,y\})$, i.e., GMP preserves satisfaction.
         
        Suppose $\pi\models(\varphi;\,x)$ and $\pi\models(\varphi\limp\psi;\,y)$, i.e., $N(\varphi)\geq x$ and $N(\varphi\limp\psi)\geq y$. Now, by Lemma \ref{lem: comp'on  min nd arrow}, we have $N(\psi)\geq \min\{N(\varphi), N(\varphi\limp \psi)\}\ge\min\{x,y\}$. Hence, $\pi\models(\psi;\,\min\{x,y\})$.
     
        \item\label{sd'ss;3} If $\pi\models(\varphi;\,x)$ then $\pi\models(\varphi;\,y)$ for all $y\leq x$, i.e., the rule S preserves satisfaction.

        Suppose $\pi\models(\varphi;\,x)$, i.e., $N(\varphi)\geq x$. Then, for any $y\le x$, $N(\varphi)\geq y$. Hence, $\pi\models(\varphi;\,y)$.
     \end{enumerate}
     
     Now, suppose $\f\vdash(\gamma;\,y)$ and $\pi$ is a possibility distribution such that $\pi\models\f$. Since $\f\vdash(\gamma;\,y)$, there is a derivation $((\gamma_1;\,y_1),\ldots,(\gamma_n;\,y_n)=(\gamma;\,y))$ of $(\gamma;\,y)$ from $\f$.

     We will now show, by induction on the length of the derivation $n$, that $\pi\models(\gamma_{i};\, y_{i})$ for all $1\leq i\leq n$, and hence, in particular, $\pi\models(\gamma;\,y)$.

     \textsc{Base case:} $n=1$

     We note that $(\gamma_1;\,y_1)$ is either an instance of a $\plog$-axiom or is a member of $\f$. If  $(\gamma_{1};\,y_{1})$ is an instance of an axiom, then, by \ref{sd'ss;1}, $\pi\models(\gamma_1;\,1)$, and hence, by \ref{sd'ss;3}, $\pi\models(\gamma_1;\,y_1)$.

     On the other hand, if $(\gamma_{1};\, y_{1})\in\f$, then the result follows from our assumption that $\pi\models\f$.

     \textsc{Induction hypothesis:} Suppose $\pi\models(\gamma_i;\,y_i)$ for all $1\le i\le k$ for some $1\le k<n$.

    \textsc{Induction step:} We need to show that $\pi\models(\gamma_{k+1};\,y_{k+1})$.

     If $(\gamma_{k+1};\,y_{k+1})\in \f$ or $(\gamma_{k+1};\,y_{k+1})$ is an instance of an axiom of $\plog$, the conclusion follows by the same arguments as in the base case. Otherwise, we have the following cases.
     \begin{enumerate}[label=(\alph*)]
         \item There exists $1\le s,t\leq k$ such that $(\gamma_{k+1};\,y_{k+1})$ is obtained by GMP from $(\gamma_s;\,y_s)$ and $(\gamma_{t};\;y_{t})=(\gamma_{s}\limp\gamma_{k+1};\,y_t)$. Then, $y_{k+1}=\min\{y_s,y_t\}$ and $\pi\models(\gamma_{k+1};\,y_{k+1})$ by \ref{sd'ss;2}.
         \item There exists $j\leq k$ such that $(\gamma_{k+1};\,y_{k+1})$ is obtained from  $(\gamma_{j};\,y_{j})$ by rule S. Then, $y_{k+1}\le y_j$ and $\pi\models(\gamma_{k+1};\,y_{k+1})$ by \ref{sd'ss;3}. 

    Thus, $\pi\models(\gamma_i;\,y_i)$ for all $1\le i\le n$, and hence, in particular, $\pi\models(\gamma;\,y)$.

    This proves that $\f\models(\gamma;\,y)$.
    \end{enumerate}
 \end{proof}

\section{Completeness}\label{sec:poslog comp}
In this section, we discuss the completeness of $\plog$ with respect to the semantics described in Section \ref{sec:poslog}. Before going into the results, we need a few definitions and results.
     \begin{dfn}
        Let $\f\subseteq\fm$ be a set of possibilistic formulas. We define the \emph{degree of triviality} of $\f$, denoted by $\tr\f$, as follows.
            \[
            \tr\f:=\sup\{x\in(0,1]\mid\,\f\models(\bot;\,x)\}=\wt(\bot,\f).
            \]
         \end{dfn}

  \begin{rem}
      The above definition has been used for inconsistency, and denoted by Incons($\f$) in \cite{DuboisLangPrade1994}, where a possibilistic logic has been defined on classical logic. This is justified by the fact that in the setting of classical logic, triviality and inconsistency are equivalent. However, negation inconsistency does not lead to triviality in LFIs, so these two notions can be distinguished in the current framework. This is discussed further in Section \ref{sec:poslog notions}.
  \end{rem}

    \begin{rem}\label{rem:triv_equalities}
        Let $\f\subseteq\fm$ be a set of possibilistic formulas. Then, 
        \[
        \tr\f=N_\f(\bot)=\inf\{N(\bot)\mid\,\pi\models\f\}=\mathrm{inf}\{1-\pi_\f(w)\mid\,w\in\val\},
        \]
        where $N_\f$ is the necessity measure induced by the possibility distribution induced by $\f$, $\pi_\f$, and $N$ in the third term above is the necessity measure induced by $\pi$.
            
        The first equality follows from the definition and Corollary \ref{cor:wt=nec'ity}. The second equality follows from Remark \ref{rem:N_F=infN}. The last equality follows by the definition of $N_\f$ and the fact that $w(\bot)=0$ for all $w\in\val$.
    \end{rem}

\begin{lem}\label{lem:Tfae}
       Suppose $\f\subseteq\fm$ is a finite set of possibilistic formulas. Then,  $\f^*$ is satisfiable iff   $\tr\f=0$. 
\end{lem}       

       \begin{proof}
        Suppose $\f^*$  is satisfiable. Then, there exists an $\mb$-valuation $w^*$ such that $w^*(\varphi)=1$ for all $\varphi\in \f^*$.  So, $\pi_\f(w^*)=\inf\{1-x\mid\,(\varphi;\,x)\in\f,w^*(\varphi)=0\}=\inf\emptyset=1$. Thus, by Remark \ref{rem:triv_equalities}, $\tr\f=\mathrm{inf}\{1-\pi_\f(w)\mid\,w\in\val\}=0$ (attained at $w^*$).

        Conversely, suppose $\tr\f=0$. If $\f=\emptyset$, then $\f^*=\emptyset$, and hence, is satisfiable. Next, let $\f\neq\emptyset$. Then, as $\tr\f=\inf\{1-\pi_\f(w)\mid\,w\in\val\}=0$  and $\f$ is finite, $1-\pi_\f(w)=0$, i.e., $\pi_\f(w)=1$, for some $w\in\val$. By definition, $\pi_\f(w)=\inf\{1-x\mid\,(\varphi;\,x)\in\f,w(\varphi)=0\}=1$. If possible, let $\{1-x\mid\,(\varphi;\,x)\in\f,w(\varphi)=0\}\neq\emptyset$. Then, since $\f$ is finite, there exists $(\varphi;\,x)\in \f$ such that $w(\varphi)=0$ and  $1-x=1$, i.e., $x=0$. This is a contradiction since weights of possibilistic formulas are in $(0,1]$. So, $\{1-x\mid\,(\varphi;\,x)\in\f,w(\varphi)=0\}=\emptyset$. This implies that $w(\varphi)\neq0$, i.e., $w(\varphi)=1$, for all $\varphi\in \f^*$. Hence, $\f^*$ is satisfiable.  
       \end{proof}

\begin{lem}\label{lem:cor'y deduction theorem}
     Suppose $\f\cup\{(\varphi;\,x)\subseteq\fm$.
     $\f\models(\varphi;\,x)$ iff $\f\cup\{(\sim\varphi;\,1)\}\models(\bot;\,x)$.
\end{lem}
 
\begin{proof}
    By the Deduction theorem (Theorem \ref{thm:ded'on thm poslog}), $\f\cup\{(\sim\varphi;\,1)\}\models(\bot;\,x)$ iff $\f\models(\sim\varphi\limp\bot;\,x)$, i.e., $\f\models(\sim\sim\varphi;\,x)$. Now, it can be easily checked that $\sim\sim\varphi\models_\mb\varphi$ and $\varphi\models_\mb\sim\sim\varphi$. Thus, by Lemma \ref{lem:equivalence}, $\f\models(\sim\sim\varphi;\,x)$ iff $\f\models(\varphi;\,x)$. Hence, $\f\models(\varphi;\,x)$ iff $\f\cup\{(\sim\varphi;\,1)\}\models(\bot;\,x)$.
\end{proof}

\begin{lem}\label{lem:wt=trv}
     Suppose $\f\subseteq\fm$ is a set of possibilistic formulas and $\varphi$ is an $\mb$-formula. Then, $\wt(\varphi,\f)=\tr(\f\cup\{(\sim\varphi;\,1)\})$.
\end{lem}

\begin{proof}
    \textsc{Case 1:} $\wt(\varphi;\,\f)=0$.

    So, $\sup\{x\in(0,1]\mid\,\f\models(\varphi;\,x)\}=0$. Since $x>0$ for any $(\varphi;\,x)$ such that $\f\models(\varphi;\,x)$, this implies that $\f\not\models(\varphi;\,x)$ for all $x\in(0,1]$. Then, by Lemma \ref{lem:cor'y deduction theorem}, $\f\cup\{(\sim\varphi;\,1)\}\not\models(\bot;\,x)$ for all $x\in(0,1]$. So, $\{x\in(0,1]\mid\, \f\cup(\sim\varphi;\,1)\models(\bot;\,x)\}=\emptyset$, and hence, $\tr(\f\cup\{(\sim\varphi;\,1)\})=\sup\{x\in(0,1]\mid\, \f\cup(\sim\varphi;\,1)\models(\bot;\,x)\}=0$. Thus, $\wt(\varphi,\f)=\tr(\f\cup\{(\sim\varphi;\,1)\})$.

    \textsc{Case 2:} $\wt(\varphi;\,\f)>0$.

    Suppose $\wt(\varphi,\f)=a$ and $\tr(\f\cup\{(\sim\varphi;\,1)\})=b$. 

     By Corollary \ref{cor:wt=nec'ity}, $\wt(\varphi,\f)=N_\f(\varphi)=a$ and $\f\models(\varphi;\,a)$, since $a>0$. Then, by Lemma \ref{lem:cor'y deduction theorem}, $\f\cup(\sim\varphi;\,1)\models(\bot;\,a)$. This implies that $a\in \{x\in(0,1]\mid\, \f\cup(\sim\varphi;\,1)\models(\bot;\,x) \}$. So, 
     \[
     a\leq \sup\{x\in(0,1]\mid\, \f\cup(\sim\varphi;\,1)\models(\bot;\,x)\}=\tr(\f\cup\{(\sim\varphi\;1)\})=b.
     \]
     Since $a>0$, this implies that $b>0$. Now, $\tr(\f\cup\{(\sim\varphi;\,1)\})=b$  implies that $N_{\f\cup\{(\sim\varphi;\,1)\}}(\bot)=b$, by Remark \ref{rem:triv_equalities}. Then, by Corollary \ref{cor:wt=nec'ity}, $\f\cup\{(\sim\varphi;\,1)\}\models(\bot;\,b)$. So, by Lemma \ref{lem:cor'y deduction theorem}, $\f\models(\varphi;\,b)$. Thus, $b\in\{x\in(0,1]\mid\,\f\models(\varphi;\,x)\}$, and so,
     \[
     b\le\sup\{x\in(0,1]\mid\,\f\models(\varphi;\,x)\}=\wt(\varphi,\f)=a.
     \]
     Hence, $\wt(\varphi,\f)=a=b=\tr(\f\cup\{(\sim\varphi;\,1)\})$.
\end{proof}

 \begin{lem}\label{lem:weight deduction}
     Suppose $\f$ is a finite set of possibilistic formulas and $(\varphi;\,x)$ is a necessity valued formula. Then, $\f\models(\varphi;\,x)$ iff $\f_{x}\models(\varphi;\,x)$.
 \end{lem}
 
 \begin{proof}
      Suppose $\f_x\models(\varphi;\,x)$. Then, since $\f_{x}\subseteq\f$, by monotonicity, $\f\models(\varphi;\,x)$.

     Conversely, suppose $\f\models(\varphi;\,x)$. Now, by Lemma \ref{lem:inducedposs'distr}, $\pi_\f\models\f$, where $\pi_\f$ is the possibility distribution induced by $\f$. Then, $\pi_\f\models(\varphi;\,x)$, which implies that  $N_\f(\varphi)\geq x$, where $N_\f$ is the necessity measure induced by $\pi_\f$. So, $\inf\{1-\pi_\f(w)\mid\,w(\varphi)=0,w\in\val\}\ge x$. So, $1-\pi_\f(w)\ge x$, for all $w\in\val$ such that $w(\varphi)=0$, which implies that $\pi_\f(w)\leq 1-x\hbox{ for all }w\in\val\hbox{ such that }w(\varphi)=0$.
    
    Moreover, as $x$ is the weight of a possibilistic formula, $x>0$. So, $\pi_\f(w)<1$ for any $w\in\val$ such that $w(\varphi)=0$. Then, for any $w\in\val$ such that $w(\varphi)=0$, $\pi_\f(w)=\inf\{1-y\mid\,(\psi;\,y)\in\f,w(\psi)=0\}\neq 1$. This implies that $\{1-y\mid\,(\psi;\,y)\in\f,w(\psi)=0\}\neq\emptyset$. Now, as $\f$ is finite, for any $w\in\val$ with $w(\varphi)=0$, there exists $(\psi;\,y)\in\f$ such that $w(\psi)=0$ and $\pi_\f(w)=1-y$.

    In other words, for any $w\in\val$ with $w(\varphi)=0$, there exists $(\psi;\,y)\in\f$ such that $w(\psi)=0$ and $\pi_\f(w)=1-y\le1-x$, i.e., $y\ge x$. Thus, for any $w\in\val$ with $w(\varphi)=0$, 
    \[
    \{1-y\mid\,(\psi;\,y)\in\f,w(\psi)=0,y\ge x\}=\{1-y\mid\,(\psi;\,y)\in\f_x,w(\psi)=0\}\neq\emptyset.
    \]
   Moreover, for any $w\in\val$ with $w(\varphi)=0$, $\inf\{1-y\mid\,(\psi;\,y)\in\f_x,w(\psi)=0\}=\pi_{\f_x}(w)\le1-x$, i.e, $1-\pi_{\f_x}(w)\ge x$. Thus, $\inf\{1-\pi_{\f_x}(w)\mid\,w(\varphi)=0,w\in\val\}=N_{\f_x}(\varphi)\ge x$, where $N_{\f_x}$ is the necessity measure induced by $\pi_{\f_x}$. So, $\pi_{\f_x}\models(\varphi;\,x)$. Hence, by Corollary \ref{cor:pi_FsatF}, $\f_{x}\models(\varphi;\,x)$.
 \end{proof}

 \begin{rem}\label{rem:derivation recons'on}
    Suppose $\f_y^*\vdash_\mb \gamma$. Then, there exists a derivation of $\gamma$ from $\f_y^*$ in $\mb$. We claim that there exists a derivation of $(\gamma;\,y)$ from $\f_y$. We can construct this derivation of $(\gamma;\,y)$, by an inductive argument, using the steps in the derivation of $\gamma$ from $\f^*_y$. A derivation in $\mb$ consists of a sequence of formulas that are either instances of tautologies,  formulas from $\f^*_y$ or obtained by the rule MP. Now, we attach weights to the formulas in this $\mb$-derivation as follows.
    
    First, to each instance of a tautology $\varphi$ in the $\mb$-derivation, we assign a weight of 1 as $\mb$-tautologies have weight 1 in $\plog$ by Lemma \ref{lem:properties}. We then apply the second inference rule of $\plog$ to obtain $(\varphi;\,y)$. 
    
    Next, for any formula $\varphi\in\f_y^*$ used in the $\mb$-derivation, we attach its original weight as given in $\f_y$. Now, since any formula in $\f_y$ has weight greater than or equal to $y$, we again apply the second inference rule to deduce the possibilistic formula $(\varphi;\,y)$. 

    Finally, suppose $\varphi$ is obtained from $\alpha$ and $\alpha\limp\varphi$ by a use of MP in the $\mb$-derivation. Then, by the inductive hypothesis, $(\alpha;\,y)$ and $(\alpha\limp\varphi;\,y)$ have been obtained in the $\plog$-derivation that is being constructed. Hence, by GMP, we get $(\varphi;\,y)$ as $\min\{y,y\}=y$.
    
    Thus, following the above procedure, we can obtain a $\plog$-derivation of $(\gamma;\,y)$ from $\f_y$.
\end{rem}

 \begin{lem}\label{lem:proj'on deduction}
    Suppose $\f\subseteq\fm$ is finite and $(\varphi;\,x)\in\fm$. Then, $\f\models(\varphi;\,x)$ iff $\f^{*}_{x}\models_\mb\varphi$.
 \end{lem}
 
\begin{proof} 
    Suppose  $\f\models(\varphi;\,x)$. Then, by Lemma \ref{lem:weight deduction}, $\f_{x}\models(\varphi;\,x)$. By Lemma \ref{lem:cor'y deduction theorem}, this implies that $\f_x\cup\{(\sim\varphi;\,1)\}\models(\bot;\,x)$. So, 
    \[
    \tr(\f_{x}\cup\{(\sim\varphi;\,1)\})=\sup\{y\in(0,1]\mid\,\f_x\cup\{(\sim\varphi;\,1)\}\models(\bot;\,y)\}\geq x>0.
    \]
    Then, by Lemma \ref{lem:Tfae}, $\left(\f_x\cup\{(\sim\varphi;\,1)\}\right)^*=\f_{x}^{*}\cup\{\sim\varphi\}$ is not satisfiable, which implies by Remark \ref{rem:unsat=trv}, that $\f_{x}^{*}\cup\{\sim\varphi\}\models_\mb\psi$ for any formula $\psi\in\lang$, i.e., $\f_{x}^{*}\cup\{\sim\varphi\}$ is trivial. Thus, $\f_{x}^{*}\cup\{\sim\varphi\}\models_\mb\bot$. By the Deduction theorem and the soundness theorem for $\mb$, we have $\f_x^*\models_\mb\,\sim\varphi\limp\bot$, i.e., $\f^{*}_{x}\models_\mb \,\sim\sim\varphi$. Since $\sim\sim\varphi\models_\mb\varphi$ and $\varphi\models_\mb\sim\sim\varphi$, we then have $\f^{*}_{x}\models_\mb \varphi$. 

    Conversely, suppose $\f^{*}_{x}\models_\mb\varphi$. Then, by the completeness of $\mb$, $\f^{*}_{x}\vdash_\mb\varphi$. So, there exists a derivation of $\varphi$ from $\f^{*}_{x}$. Thus, by Remark \ref{rem:derivation recons'on}, there exists a derivation of $(\varphi;\,x)$ from $\f_x$, i.e., $\f_x\vdash(\varphi;\,x)$. Hence, by soundness of $\plog$ (Theorem \ref{thm:soundness}), $\f_x\models (\varphi;\,x)$. Then, as $\f_x\subseteq \f$, by monotonicity, $\f\models(\varphi;\,x)$. 
\end{proof}
     
\begin{thm}[Completeness]
    Suppose $\f\subseteq\fm$ is finite and $(\gamma;\,y)\in\fm$. Then, $\f\models(\gamma;\,y)$ implies that $\f\vdash(\gamma;\,y)$.
\end{thm}
\begin{proof}
    Suppose $\f\models(\gamma;\,y)$. Then, by Lemma \ref{lem:proj'on deduction}, $\f^{*}_{y}\models_\mb\gamma$, and hence, by the completeness of $\mb$, $\f^*_y\vdash_\mb\gamma$. So, there exists an $\mb$-derivation of $\gamma$ from  $\f^{*}_{y}$. This implies, by Remark \ref{rem:derivation recons'on}, that there exists a $\plog$-derivation of $(\gamma;\,y)$ from $\f_{y}$. Thus, $\f_y\vdash(\gamma;\,y)$, and so, by monotonicity, $\f\vdash(\gamma;\,y)$. 
\end{proof}

\begin{rem}
    The above completeness theorem depends on the finiteness of the set $\f$ of possibilistic formulas. This condition can be relaxed to infinite $\f\subseteq\fm$, provided the corresponding set of weights is finite. 
\end{rem}

\section{Contradictoriness and consistency measures}\label{sec:poslog notions}
     The consistency measure of a set of formulas in a possibilistic logic is defined in \cite{DuboisLangPrade1994}, where the underlying logic is the classical propositional logic. In this article, the underlying logic $\mb$ has a consistency operator in the object language. So, we will separately introduce consistency measures at the meta- and the object-language levels.

     \begin{dfn}
    \emph{Consistency at the meta level} of a set of possibilistic formulas $\f\subseteq\fm$, denoted by $\cm\f$, is defined as follows.
    \[
    \cm\f=\sup\{\pi_\f(w)\mid\, w\in\val\}.
    \]     
    \end{dfn}

    As the definition suggests, $\cm$ is the consistency of a set of formulas at the meta-language level. This concept is similar to the one present in literature. However, since the underlying logic here is $\mb$, we introduce the following new notion of consistency at the object level.  Consistency of a formula $\alpha$ in an LFI is usually depicted in object-language with $\circ \alpha$. So we can define consistency of a set of formulas at the object-language level using $\circ\alpha$ as follows.

    \begin{dfn}
          \emph{Consistency at the object level} of a set of possibilistic formulas $\f\subseteq\fm$, denoted by $\co\f$, is defined as follows.
          \[
          \co\f=\mathrm{sup}\{x\in(0,1]\mid\,\f\models(\circ\alpha;\,x),\mbox{ where } \alpha \in\f ^*\mbox{ or }\neg \alpha \in\f^*\}
          \]
      \end{dfn}

\begin{rem}
  Given a set $\f\subseteq\fm$, the definition of $\co\f$ requires either $\alpha$ or $\neg\alpha$ to be in $\mathcal{F}^*$. The justification behind having this requirement is illustrated by the following example. Suppose $\mathcal{F} = \{(\neg p; 0.6), (\circ p; 0.7)\}$. If the disjunctive clause $\neg\alpha \in \mathcal{F}^*$ were excluded, then we would not be able to evaluate $\wt(\circ p,\f)=\co\f$, even though $\neg p \in \mathcal{F}^*$. In other words, despite $\mathcal{F}$ containing information about the falsity of $p$, information about the consistency of $p$ (i.e., $\circ p$) would be completely missed, which is undesirable.
  
  $\co\f$ gives the degree of consistency of propositions that are actually present in $\f$. For example, if $\f=\{(p;\,0.8), (\circ q;\,0.6)\}$, then since neither $q$ nor $\neg q$ belongs to $\f^*$, $\wt(\circ q,\f)$ will not be considered in the evaluation of $\co\f$. 
 
  If $\co\f>0$, then it indicates that there is some piece of information (i.e., some $\alpha$ or $\neg \alpha$) in $\f$ such that information about its consistency (i.e., $\circ \alpha$), along with its weight, is derivable from $\f$.   
    \end{rem}
    
    \begin{dfn}
        Contradictoriness of a set of possibilistic formulas $\f\subseteq\fm$, denoted by $\ctr\f$ is defined as follows.
        \[
          \ctr\f=\mathrm{sup}\{x\in(0,1]\mid\,\f\models(\alpha\land\neg\alpha;\,x),\mbox{ where } \alpha \in\f ^*\mbox{ or }\neg \alpha \in\f^*\}.
        \]
    \end{dfn}
     
    \begin{thm}  
        Suppose $\f\subseteq\fm$. Then, $\ctr\f\geq\mathrm{Trv}\;\f$.
    \end{thm}

    \begin{proof}
        If $\tr\f=0$, then we have the desired result. This covers the case when $\f=\emptyset$, since in that case, $\tr\f=\sup\emptyset=0$.

        So, now, let $\f\neq\emptyset$ and $\tr\f=a>0$. Then, $\wt(\bot,\f)=a$, and hence, by Corollary \ref{cor:wt=nec'ity}, $\f\models(\bot;\,a)$. Now, $\f\neq\emptyset$ implies that $\f^*\neq\emptyset$. Let $\alpha\in\f^*$. Since $\bot\models_\mb\alpha\land \neg\alpha$, by Lemma \ref{lem: nec'ty axioms proof}, $N_\f(\alpha\land \neg\alpha)\geq N_\f(\bot)=\tr\f=a$, where $N_\f$ is the necessity measure induced by $\pi_\f$, the possibility distribution induced by $\f$. This implies that $\pi_\f\models(\alpha\land\neg\alpha;\,a)$, and so, by Corollary \ref{cor:pi_FsatF}, $\f\models(\alpha\land\neg\alpha;\,a)$. Thus, $a\in \{x\in(0,1]\mid\,\f\models(\alpha\land\neg\alpha;\,x),\mbox{ where } \alpha \in\f ^*\mbox{ or }\neg \alpha \in\f^*\}$. Hence, $\tr\f=a\leq \mathrm{sup}\{x\in(0,1]\mid\,\f\models(\alpha\land\neg\alpha;\,x),\mbox{ where } \alpha \in\f ^*\mbox{ or }\neg \alpha \in\f^*\}=\ctr\f$. 
    \end{proof}
    
    \begin{exa}\label{exa:poslog general example} 
        Suppose $\f=\{(p;\,a_1),(\neg p;\,a_2),(q;\,a_3)\}$, where $p,q\in V$ are variables of $\mb$, and $a_1,a_2,a_3\in(0,1]$. 
        We now compute $\co\f,\ctr\f,\cm\f$, and $\tr\f$ as follows.
    
      The following table lists the possible valuations for $p,q\in V$.
     \begin{center}
     $\begin{array}{|c|c|c|c|c|}
           \hline p& \neg p & q&\circ p&\\\hline
             1&1  &1&0&v_1 \\\cline{2-2}
             &0  & &1&v_2\\\hline
            1& 1 & 0&0&v_3\\\cline{2-2}
             &0  & &1&v_4\\\hline
              0&1  &1 &1&v_5\\\hline
             0&1  & 0&1&v_6\\\hline
    \end{array}$
    \end{center}
     We now compute the possibility distribution induced by $\f$, i.e., $\pi_\f$, using Definition \ref{def:posdist'on on formulae}, as follows.
      
      $\pi_\f(v_1)=1$, since $v_1(p)=v_1(\neg p)=v_1(q)=1$.
      
      $\pi_\f(v_2)=1-a_2$, since $v_2(\neg p)=0$ and $(\neg p;\,a_2)\in\f$. By similar reasoning, $\pi_\f(v_3)=1-a_3$, and $\pi_\f(v_5)=1-a_1$.
      
      $\pi_\f(v_4)=\min\{1-a_2,1-a_3\}$, since $v_4(\neg p)=v_4(q)=0$ and $(\neg p;\,a_2),(q;\,a_3)\in\f$. By similar reasoning, $\pi_\f(v_6)=\min\{1-a_1,1-a_3\}$.

      Now, 
      \[
      \begin{array}{lcl}
           \co\f&=&\sup\{x\in(0,1]\mid\,\f\models(\circ\alpha;\,x),\mbox{ where } \alpha \in\f ^*\mbox{ or }\neg \alpha \in\f^*\}\\
           &=&\max\{\wt(\circ p,\f),\wt(\circ\neg p,\f),\wt(\circ q,\f)\}\\
           &=&\max\{N_\f(\circ p),N_\f(\circ\neg p),N_\f(\circ q)\},
      \end{array}
     \]
     where $N_\f$ is the possibility measure induced by $\pi_\f$. The last equality above is obtained by using Corollary \ref{cor:wt=nec'ity}.
     
     Now, $N_{\f}(\circ p)=\inf\{1-\pi_\f(w)\mid\,w(\circ p)=0,w\in\val\}=\min\{1-\pi_\f(v_1),1-\pi_\f(v_3)\}$. Thus, $N_\f(\circ p)=\min\{0,a_3\}=0$. Next, we note that there exists a valuation $v$ such that $v(p)=v(\neg p)=v(q)=v(\neg\neg p)=1$. This implies that $v(\circ\neg p)=0$ and $\pi_\f(v)=1$. So, $N_{\f}(\circ \neg p)=0$. Similarly, $N_\f(\circ q)=0$, as there exists a valuation $v^\prime$ such that $v^\prime(p)=v^\prime(\neg p)=v^\prime(q)=v^\prime(\neg q)=1$, and thus, $v^\prime(\circ q)=0$ and $\pi_\f(v^\prime)=1$.  

    Hence, $\co\f=0$, which implies that $\f$ does not satisfy $(\circ p;\,x)$, $(\circ\neg p;\,x)$ and $(\circ q;\,x)$ for any $x>0$. 

    \[
    \begin{array}{lcl}
         \ctr\f&=&\sup\{x\in(0,1]\mid\,\f\models(\alpha\land\neg\alpha;\,x),\mbox{ where } \alpha \in\f^* \mbox{ or }\neg \alpha \in\f^*\}\\
         &=&\max\{\wt(p\land\neg p,\f),\wt(\neg p\land\neg\neg p,\f),\wt(q\land\neg q,\f)\}\\
         &=&\max\{N_\f(p\land\neg p),N_\f(\neg p\land\neg\neg p),N_\f(q\land\neg q)\},
    \end{array}
    \]
    where $N_\f$ is as mentioned before. As before, the last equality is obtained by using Corollary \ref{cor:wt=nec'ity}. 
 
    Now, 
    \[
    \begin{array}{lcl}
    N_{\f}(p\land \neg p)&=&\inf\{1-\pi_\f(w)\mid\,w(p\land\neg p)=0,w\in\val\}\\
    &=&\min\{1-\pi_\f(v_2),1-\pi_\f(v_4),1-\pi_\f(v_5),1-\pi_\f(v_6)\}\\
    &=&\min\{a_2,1-\min\{1-a_2,1-a_3\},a_1,1-\min\{1-a_1,1-a_3\}\}\\
    &=&\min\{a_2,\max\{a_2,a_3\},a_1,\max\{a_1,a_3\}\}\\
    &=&\min\{a_2,a_1\}.
    \end{array}
    \]
    
    We note that there exists a valuation $v$ such that $v(p)=v(\neg p)=v(q)=1$ and $v(\neg\neg p)=0$, which implies that $v(\neg p\land\neg\neg p)=0$ and $\pi_\f(v)=1$. So, $N_{\f}(\neg p\land \neg \neg p)=0$. By similar arguments, $N_\f(q\land\neg q)=0$. Hence, $\ctr \f=\min\{a_1,a_2\}>0$.
       
    $\cm\f=\sup\{\pi_\f(w)\mid\,w\in\val\}=1$, since $\pi_\f(v_1)=1$.
      
    $\tr\f=\inf\{1-\pi_\f(w)\mid\,w\in\val\}=0$, since $\pi_\f(v_1)=1$.

    Thus, we can say that $\f$ is a set of possibilistic formulas which is contradictory, since $\ctr\f>0$ but not trivial, as $\tr\f=0$. This illustrates the difference between contradictoriness and triviality measures of sets of $\plog$-formulas. 
    
    The difference between the meta and the object level consistencies of $\f$ can also be observed in this example. $\cm\f=1$, which represents the maximum value of the canonical possibility distribution satisfying $\f$, while $\co\f=0$. This difference can be attributed to the fact that in LFIs, the notion of consistency is formalized at the object level. Since no formula of the form $\circ \alpha$ is entailed from $\f$ for an $\alpha\in\lang$ such that either $\alpha$ or $\neg\alpha$ is present in $\f$, $\co\f=0$.  
\end{exa}

The notions discussed in this article are useful when there are multiple sources of information that are conflicting, uncertain, or unreliable. We can, for instance, measure the strength of evidence by possibility/necessity measures, quantify conflict by the degree of contradictoriness, differentiate contradiction from total failure, which can be given by the degree of triviality. Thus, there can be applications in situations where several sensors are used, for example, in autonomous vehicles, aircraft navigation systems, or fields where decisions are based on the opinions of multiple experts, such as medical diagnosis and legal decision support. In these scenarios, our framework allows for more nuanced judgment than a binary conclusion. Such a scenario is illustrated in the following example.

\begin{exa}
An autonomous vehicle is broadly defined as one equipped with technology that senses the conditions around it, including traffic, pedestrians, and physical hazards and can adjust its course and speed without a human at the controls. The terms autonomous and self-driving cars are often used interchangeably. They mainly use cameras, radar (radio detection and ranging) and LiDAR (Light Detection and Ranging) to analyze surroundings, obstacles, control steering and braking. We will not go into the details of how these work. Here we are interested in situations where conflicting readings are provided by different sensors that are used by such a car.

Let $p$ and $q$ denote the propositions `a pedestrian is in front of the car' and `the traffic light is red,' respectively. Suppose the camera reading says that a pedestrian is detected with certainty 0.85, the radar says that no pedestrian is detected with certainty 0.75, and the LiDAR says that a red light is detected with certainty 0.9. So, the situation can be represented by the set of $\plog$ formulas $\f=\{(p;\,0.85), (\neg p;\,0.75), (q;\,0.9)\}$. We note that the corresponding set of $\mb$-formulas $\f^*$ is as in Example \ref{exa:poslog general example}. Thus, we have the same valuations as in there. By substituting $a_1=0.85,a_2=0.75$, and $a_3=0.9$, we have $\pi_\f(v_1)=1$, $\pi_\f(v_2)=1-a_2=0.25$, $\pi_\f(v_3)=1-a_3=0.1$, $\pi_\f(v_4)=\min\{1-a_2,1-a_3\}=0.1$, $\pi_\f(v_5)=1-a_1=0.15$, and $\pi_\f(v_6)=\min\{1-a_1,1-a_3\}=0.1$. 

Then, $\ctr\f=\min\{a_2,a_1\}=0.75$, $\tr\f=0$, $\cm\f=1$, and $\co\f=0$.

The contradictoriness of the given situation, measured by $\ctr\f$, comes out to be 0.75. This shows that there is strong contradictory evidence about the presence of a pedestrian in front of the car. This is clearly more informative than merely concluding that the set is contradictory.

$\tr\f=0$ shows that despite the presence of contradictory evidence, the system does not collapse and lead to triviality. The car can still reason about traffic lights and take necessary action.  This highlights the main goal of this work, that is, to have a system that, despite contradictory readings, will not stop or make absurd conclusions.

$\cm\f=1$ expresses the fact that there exists at least one highly plausible valuation of the sensor data. So, the reading set remains usable. 

$\co\f=0$ shows that there is no conclusive evidence that these readings are mutually consistent. 

An autonomous controller could use these measures instead of just a certainty measure $N$. It could take action according to conditions such as, if $\ctr\f\geq 0.5$, then engage emergency caution mode, or that, if $\tr\f>0$, then show sensor system failure alert.

\section{Conclusions}
In this article, we have explored the intersection of possibility theory and LFIs. We have observed that neither LFIs nor the classical possibilistic framework alone can capture the situation that involves inconsistency and uncertainty at the same time. This is exemplified by the discussion in the previous section, where practical scenarios with uncertainty and inconsistency have been presented. In these situations, we clearly cannot afford a total system collapse. 

LFIs allow reasoning in the presence of contradictions but do not distinguish among different evidence available. Classical possibilistic framework associates weights with information to represent degree of certainty, but under classical logic, any contradiction leads to explosion, making the system incapable of working in such situations. Thus, by integrating possibility theory with LFIs, our framework not only prevents explosion but also measures the degree of conflicting evidence, which is beneficial. This framework is achieved via the system $\plog$, which is a possibilistic logic on the LFI $\mb$.

\end{exa}
      \bibliography{PosLog.bib}
      \bibliographystyle{plain}
\end{document}